\documentclass[11pt,a4]{article}
\usepackage[colorlinks=true, pdfstartview=FitV, linkcolor=blue, citecolor=blue,
urlcolor=blue]{hyperref}
\usepackage[applemac]{inputenc}
\usepackage{amssymb,amsmath,amsthm,amsfonts,amscd,amstext}
\usepackage[mathscr]{eucal}
\usepackage{multirow}
\usepackage{xypic}
\usepackage{array}
\usepackage{graphicx}
\newtheorem{thm}{Theorem}[section]
\newtheorem{lem}[thm]{Lemma}
\newtheorem{prop}[thm]{Proposition}
\newtheorem{cor}[thm]{Corollary}
\newtheorem{df}[thm]{Definition\rm}
\newtheorem{rem}{\it Remark\/}
\makeatletter
\providecommand*{\shuffle}{%
  \mathbin{\mathpalette\shuffle@{}}%
}
\newcommand*{\shuffle@}[2]{%
  \sbox0{$#1\vcenter{}$}%
  \kern .15\ht0 
  \rlap{\vrule height .25\ht0 depth 0pt width 2.5\ht0}%
  \raise.1\ht0\hbox to 2.5\ht0{%
    \vrule height 1.75\ht0 depth -.1\ht0 width .17\ht0 %
    \hfill
    \vrule height 1.75\ht0 depth -.1\ht0 width .17\ht0 %
    \hfill
    \vrule height 1.75\ht0 depth -.1\ht0 width .17\ht0 %
  }%
  \kern .15\ht0 
}
\makeatother

\def\dim{\operatorname{dim}}  

\def\ker{\operatorname{ker}}  

\def\N{{\mathbb N}}    

\begin{document}
\title{\bf Types, Codes and TFTs}
\author
{Roland Friedrich\thanks{Saarland University, friedrich@math.uni-sb.de}}
\maketitle
\begin{abstract}
We establish a relation between fully extended $2$-dimensional TQFTs and  recognisable weighted formal languages, rational biprefix codes and lattice TFTs. We show the equivalence of $2D$ closed TFTs and rational exchangeable series and we discuss the important special case of finite groups. Finally, we outline a reformulation in terms of a restricted version of second order monadic logic.   
\end{abstract}
\tableofcontents
\section{Preliminaries}
Let $k$ be a field of characteristic zero. We consider the following categories: $\mathbf{Alg}_k$, of associative and unital $k$-algebras, $\mathbf{cAlg}_k$, associative, commutative and unital $k$-algebras and $\mathbf{fGrp}$, of finite groups. We assume all algebras to be unital and, if necessary, $k$ also to be algebraically closed, which we shall indicate.
\subsection{Rational series}
We recall a few facts about formal series and rational languages from~\cite{BPR2009, BR,Fl1974,R1980}.
Let $X$ be a finite set, called alphabet, and $X^*$ the free monoid generated by $X$, i.e.  the set of all possible words $w$ over $X$, including the empty word $1$. A map $S:X^*\rightarrow k$, $w\mapsto S(w)=:S_w$ is called a formal series, and the $k$-algebra of all formal series is $k\langle\langle X\rangle\rangle$. The subalgebra of polynomials is $k\langle X\rangle$. Equivalently, $k\langle X\rangle$ is the free associative algebra generated by $X$ and $k\langle\langle X\rangle\rangle$ its linear dual, i.e. every formal series $S$ defines a linear functional $S:k\langle X\rangle\rightarrow k$ by $k$-linear extension of the map $w\mapsto S_w$.

The {\bf free commutative monoid} $X^+$ is defined as $X^*/\sim$ where the congruence $\sim$ is generated by the relation $xy\sim yx$ for $x,y\in X$, $x\neq y$.
Let $\alpha:X^*\rightarrow X^+$ denote the canonical epimorphism from the free non-commutative monoid $X^*$ onto the free commutative monoid $X^+$, where $\alpha(w)$ denotes the commutative image of a word $w\in X^*$. Then, by abuse of notation, $\alpha:k\langle\langle X\rangle\rangle\rightarrow k[[X]]$ is the induced extension. 

The notion of an {\bf exchangeable series} was introduced by M. Fliess~\cite{Fl1974} and it is closely related to the notion of exchangeable stochastic processes and Markov chains.
\begin{df}
\label{exser}
Let $S\in k\langle\langle X\rangle\rangle$. $S$ is called {\bf exchangeable} if for all $v,w\in X^*$, with $\alpha(v)=\alpha(w)$, $S_v=S_w$ holds.
\end{df}

\begin{df}
For $A\in\mathbf{Alg}_k$, $X$ a finite set and $S\in k\langle\langle X\rangle\rangle$,  
\begin{itemize}
\item 
the {\bf syntactic ideal} of $S$, denoted $I_S$, is defined as
$$
I_S:=\sup_{J\subset\ker(S)}\{ J~|~\text{two-sided ideal~}\},
$$
i.e. it is the {\em maximal two-sided ideal} contained in $\ker(S)\subset k\langle X\rangle$. 
\item The algebra 
$$
\mathfrak{A}_S:= k\langle X\rangle / I_S,
$$
is called the {\bf syntactic algebra} of $S$, with $\pi:k\langle X\rangle\rightarrow \mathfrak{A}_S$
the {\bf canonical algebra epimorphism}.
\item If $A\cong\mathfrak{A}_S$, for some $S$ then $A$ is called {\bf syntactic}.
\item If $A\cong\mathfrak{A}_S$, for some $S$ and $\dim_k(\mathfrak{A}_S)<\infty$, then we call $A$ {\bf rational syntactic}.
\end{itemize}
\end{df}
A codimension one linear subspace $H$ of a $k$-vector space $V$ is called a {\bf hyperplane}, i.e. $\dim_k(V/H)=1$. In particular $H$ contains the origin, i.e. $0\in H$. For the classical relation between linear functionals, their nullspaces and hyperplanes, cf. e.g. the monograph~\cite{BP2012}.

The following fundamental statements have originally been established by Ch. Reutenauer and correspond to [\cite{R1980} Proposition I.2.4 and Théorème II.1.2.]
\begin{thm}[Reutenauer~\cite{R1980}]
\label{Syntactic_hyperplane}
Let $X$ be a finite alphabet and $A\in\mathbf{Alg}_k$, a {\rm finitely generated} algebra, $A\neq0$. Then the following statements hold:
\begin{enumerate}
\item $A$ is syntactic iff there exists a hyperplane $H\subset A$, which contains no two-sided ideal other than $\{0\}$.
\item A formal series $S\in k\langle X\rangle^*$ is rational iff $\dim_k(\mathfrak{A}_S)<\infty$, i.e.  the syntactic algebra $\mathfrak{A}_S$ of $S$  is a finite-dimensional $k$-vector space.
\end{enumerate}  
\end{thm}
In the proof of the first statement above  the following auxiliary facts are used which we shall spell out explicitly. 
\begin{lem}
\label{surjective_image}
If $A,B\in\mathbf{Alg}_k$, $J\subset A$ a two-sided ideal and $\varphi:A\rightarrow B$ a surjective algebra morphism then $\varphi(J)\subset B$ is a two-sided ideal.
\end{lem}

\begin{lem}
\label{finite_alphabet}
For $A\in\mathbf{Alg}_k$, $A\neq0$, let $A$ be syntactic  and finitely generated, i.e. $A=k[a_1,\dots,a_n]$, for some $a_i\in A$, $i=1,\dots, n$. Then there exists a formal series $S\in k\langle x_1,\dots, x_n\rangle^*$ such that $A\cong k\langle x_1,\dots, x_n\rangle/I_S$, i.e. the representing alphabet $X$ can be assumed in this case to be finite. 
\end{lem}
\begin{proof}
As $A$ is syntactic, by~Theorem~\ref{Syntactic_hyperplane} 1., there exists a hyperplane $H\subset A$ and a linear functional $\lambda_H:A\rightarrow k$, induced by $H$, such that $\ker(\lambda_H)=H$. Let $j_a:k\langle x_1,\dots,x_n\rangle\rightarrow A$ be the algebra morphism onto $A$, defined by $j_a(x_i):=a_i$, $i=1,\dots,n$. Define the linear functional $S:k\langle x_1,\dots,x_n\rangle\rightarrow k$ by $S:=\lambda_H\circ j_a$ with corresponding syntactic ideal $I_S$, as shown in the diagram below:

\[
\begin{xy}
  \xymatrix{
      {k\langle x_1,\dots, x_n\rangle}\ar@/^2pc/[rr]^{{S}}\ar[d]^{\pi}\ar@{->>}[r]^{j_a}& {A=k[ a_1,\dots,a_n]}\ar[r]^{\qquad\lambda_H}&k \\
      \frac{k\langle x_1,\dots, x_n\rangle}{\ker(j_a)}\ar[ru]_{\cong}&&
                            }
\end{xy}
\]
Now, $\ker(j_a)$ is a two-sided ideal, $\pi$ the canonical algebra epimorphism and by the ``First Isomorphism Theorem for Rings", it follows that $k\langle x_1,\dots, x_n\rangle /\ker(j_a)\cong A$.
Let us show that 
\begin{equation}
\label{equal_ideals}
\ker({j_a})=I_{S},
\end{equation}
from which 
$$
\mathfrak{A}_S=\frac{k\langle x_1,\dots,x_n\rangle}{I_S}=\frac{k\langle x_1,\dots,x_n\rangle}{\ker(j_A)}\cong A,
$$
and hence, the claim follows. Let us show the equality~(\ref{equal_ideals}). 

We have $\ker(S)=j_a^{-1}(\lambda_H^{-1}(0))=j_a^{-1}(H)$ and hence $j_a^{-1}(0)=\ker(j_a)\subset I_S\subset\ker(S)$. 

By Lemma~\ref{surjective_image}, $j_a(I_{S})\subset H$ is a two-sided ideal and by assumption $H$ does not contain a non-trivial two-sided ideal and hence  $j_a(I_S)=0$. Therefore $I_S\subset\ker(j_a)$ from which~(\ref{equal_ideals}) follows.
\end{proof}
\subsection{Codes}
We recall the following facts from~\cite{BPR2009,BR, R1982}. Every lanugage $L\subset X^*$  over $X$ induces a linear functional $\chi_L$, the {\bf characteristic series} of $L$, defined as 
$\chi_L:k\langle X\rangle\rightarrow k$, by $k$-linear extension of the map $w\mapsto 1$ if $w\in L$ and $0$ otherwise.

Let $X$ be a finite alphabet and $C$ a language, i.e. $C\subset X^*$. Then $C$ is a {\bf code} if the submonoid $C^*$ generated by $C$ is free, with base $C$. The {\bf syntactic algebra} $\mathfrak{A}_{C^*}$ of a code $C$ is the syntactic algebra of $C^*$. A code $C$ is called {\bf rational} if its syntactic algebra satisfies $\dim_k(\mathfrak{A}_{C^*})<\infty$, i.e. it is a finite-dimensional $k$-vector space.

A language $L\subset X^*$ is called {\bf prefix} if for all $u,v\in X^*$ with $u,uv\in L$ it follows that $v=1$, and similarly, $L$ is called {\bf suffix}, if $v,uv\in L$ implies $u=1$. The language $L$ is called {\bf biprefix} if it is both prefix and suffix.

\begin{thm}[\cite{R1982} Theorem 1]
If a code $C$ is rational and biprefix then its syntactic algebra $\mathfrak{A}_{C^*}$, is a finite-dimensional semi-simple $k$-algebra.
\end{thm}
\subsection{Frobenius algebras}
Let us recall the following classes of algebras, cf. e.g.~\cite{L1999} and also~\cite{BR,LPf2006}.
\begin{df}
\label{simple_Frob}
Let $A\in\mathbf{Alg}_k$ and $\dim_k(A)<\infty$. Then $A$ is
\begin{enumerate}
\item {\bf simple} if $\{0\}$ and $A$ are the only two-sided ideals in $A$,
\item {\bf semi-simple} if $A\cong M_{n_1}(k)\times\dots\times M_{n_r}(k)$, $r,n_i\in\N^*$, $i\in\{1,\dots,r\}$, or alternatively,  if $\mathfrak{l}\leq A$ is a nilpotent left ideal, then $\mathfrak{l}=0$.
\item {\bf Frobenius} if it satisfies one of the following equivalent characterisations:
\begin{itemize}
\item There exists a bilinear form $B:A\times A\rightarrow k$ which is non-degenerate and associative, i.e. which satisfies
$$
B(ab,c)=B(a,bc)\qquad\forall a,b,c\in A.
$$
\item There exists a linear functional $\lambda:A\rightarrow k$ whose kernel $\ker(\lambda)$ (nullspace)  contains no left or right ideal other than zero.
\item There exists a hyperplane $H\subset A$, $(0\in H)$, which contains no nonzero right ideal (left ideal).
\end{itemize}
\item {\bf symmetric Frobenius} if $\lambda(ab)=\lambda(ba)$, for all $a,b\in A$, i.e. $\lambda$ is a {\bf trace}.
\end{enumerate}
\end{df}

\begin{rem}
If $A$ is Frobenius then $\ker(\lambda)$ does not contain a non-trivial {\em two-sided} ideal as every two-sided ideal is both a left and right ideal.
\end{rem}
\begin{prop}
\label{Semisimple-Frob}
If $A\in\mathbf{Alg}_k$, $\dim_k(A)<\infty$ and $A$ is semi-simple then it can be endowed with the structure of a {\bf symmetric Frobenius algebra}.
\end{prop}
\begin{proof}
The Wedderburn-Artin Theorem is an essential part of the proof; for the remaining details 
cf.~[\cite{L1999} Exercise 12. p. 114] or [\cite{K2003}, Exercise 9. p. 106].
\end{proof}
\subsection{Bicategories}
For this subsection we use as references~\cite{HSchV2016,L1999,Lu2009,SP2014}. \begin{df}
The symmetric monoidal bicategory $\mathbf{Alg}_k^2$ over $k$ is given by the following data: 
\begin{description}
\item[Objects:] $\mathbf{Alg}_k$, i.e. associative and unital $k$-algebras,
\item[$1$-morphisms:] $(A,B)$-bimodules, $A,B\in\mathbf{Alg}_k$, and composition is given by the tensor producut of bimodules, 
\item[$2$-morphisms:] bimodule morphisms.
\end{description}
\end{df}

In [\cite{Lu2009} 2.3.] J. Lurie discusses the notion of fully dualisable objects in a symmetric monoidal $(\infty,n)$-category. For $A\in\mathbf{Alg}_k^2$, $A$ is called {\bf fully dualisable} if it is separable.
\begin{prop}[\cite{Lu2009,SP2014}]
Let $k$ be algebraically closed. Then the fully dualisable objects ${\mathbf{fdAlg}}_k^2$  in $\mathbf{Alg}^2_k$ correspond to the finite-dimensional semi-simple $k$-algebras.
\end{prop}

The notion of {\bf Morita contexts} is discussed in~[\cite{L1999} {\S}18C], [\cite{SP2014} 3.8.4] or [\cite{HSchV2016} Section 2.].

\begin{df}
The bicategory $\mathbf{sFrob}_k$ of {\bf semi-simple symmetric Frobenius algebras} is given by the following data:
\begin{description}
\item[Objects:] semi-simple, symmetric Frobenius algebras,
\item[$1$-morphisms:] compatible Morita contexts,  
\item[$2$-morphisms:] isomorphisms of Morita contexts.
\end{description}
\end{df}

\subsection{Automaticity}
Here we show that the algebras we considered arise as syntactic algebras of recognisable power series, i.e. there exists a weighted finite-state automaton which recognises any such formal series. 
\begin{prop}
\label{rationalsyntactic_algebras}
Let $A\in\mathbf{Alg}_k$ and $\dim_k(A)=:N<\infty$. Then, if $A$ is 
\begin{enumerate}
\item  {\bf simple}, or 
\item {\bf semi-simple}, or
\item {\bf Frobenius}  
\end{enumerate}
then $A$ is {\bf rational syntactic}, i.e. there exists a recognisable series $S\in k\langle x_1,\dots,x_N\rangle^*$, such that $A\cong k\langle x_1,\dots,x_N\rangle/I_S$.
\end{prop}
\begin{rem}
In~[\cite{R1980} Examples 1., p. 452] Ch. Reutenauer lists simple and (symmetric) Frobenius algebras as examples for syntactic algebras. However, the stronger statements in Proposition~\ref{rationalsyntactic_algebras} seem to be absent from the literature. 
\end{rem}
\begin{proof}
Let us first show 3. By assumption $A=ka_1\oplus\dots\oplus ka_N$, and hence $\{a_1,\dots, a_N\}$ is a finite generating set for $A$. Therefore, the claim follows from Lemma~\ref{finite_alphabet}. In more detail:
define the surjective algebra morphism $j_a:k\langle x_1,\dots,x_N\rangle\rightarrow A$ by $x_i\mapsto a_i$, for $i=1,\dots,N$ and the linear functional $S:=\lambda\circ j_a$, where $\lambda:A\rightarrow k$ is the Frobenius form as in Definition~\ref{simple_Frob}~3., cf. the diagram below:
\[
\begin{xy}
  \xymatrix{
      {k\langle x_1,\dots, x_N\rangle}\ar@/^2pc/[rr]^{{S}}\ar[d]^{\pi}\ar@{->>}[r]^{j_a}& \bigoplus_{i=1}^N ka_i\ar[r]^{\quad\lambda}&k\\
      \frac{k\langle x_1,\dots, x_N\rangle}{\ker(j_a)}\ar[ru]_{\cong}&&
                            }
\end{xy}
\]
Then as in the proof of Lemma~\ref{finite_alphabet}, we have $\ker(j_a)=I_S$ and hence $\mathfrak{A}_S\cong A$ which shows finite dimensionality. Therefore, by Theorem~\ref{Syntactic_hyperplane}, the formal series $S$ is rational.

In order to show 1. and 2. we remark that simple implies semi-simple. Then  Proposition~\ref{Semisimple-Frob} shows that any such algebra can be endowed with the structure of a symmetric Frobenius algebra from which the claim follows from 3.
\end{proof}
The following is an opposite statement for commutative syntactic algebras.
\begin{prop}
Let $X$ be a finite set and $S\in k\langle\langle X\rangle\rangle$. If $\dim_k(\mathfrak{A}_S)<\infty$ and $\mathfrak{A}_S\in\mathbf{cAlg}_k$ then $\mathfrak{A}_S\in\mathbf{cFrob}_k$, i.e. it is a commutative Frobenius algebra. \end{prop}
\begin{proof}
By assumption, $\mathfrak{A}_S$ is rational-syntactic and by Proposition~\ref{Syntactic_hyperplane}, there exists a hyperplane $H\subset\mathfrak{A}_S$ which contains no nontrivial two-sided ideal. As $\mathfrak{A}_S$ is commutative every ideal, left or right, is two-sided. Therefore by Definition~\ref{simple_Frob}, 3., the claim follows.
\end{proof}
Types of languages form pseudo-varieties, and by extension of the {\bf Eilenberg Variety Theorem}, they correspond to pseudo-varieties of finite algebras~[\cite{R1980} Théorème III.1.1.]

It is shown in~[\cite{R1980} 2. Exemples de variétés c. p. 472] that finite-dimensional commutative syntactic algebras correspond to exchangeable rational series, cf. Definition~\ref{exser}. Therefore we have
\begin{cor}
\label{ex_cFrob}
There is a (bijective) correspondence between 
$\mathbf{cFrob}_k$ and the $s$-variety of rational exchangeable series.
\end{cor}

\begin{rem}
The above statement implies that the tangent space to a Frobenius manifold can be considered as having a weighted finite-state automaton located at every point of the manifold whose associated syntactic algebra corresponds to the Frobenius algebra at that point.
\end{rem}

\section{Fully extended two-dimensional TFTs}
Previously with T. Kato~\cite{FK2011} we outlined how several (classes) of phenomena related to enumerative problems in geometry or integrable systems are principally governed by cellular automata. Here we relate weighted finite-state automata to open-closed string theories.

Throughout this section we assume $k$ to be algebraically closed and of characteristic zero. 

(Oriented) open-closed cobordisms, have been studied e.g. by  A. Lauda and H. Pfeiffer~\cite{LPf2006} with the help of {\bf knowledgable Frobenius algebra} which they introduced, and J. Morton~\cite{M2007} considered weak $2$-functors from $\mathbf{nCob}_2$ to $\mathbf{Vect}_2$. J. Lurie's fundamental work~\cite{Lu2009} aims at classifying all TFTs which he achieves by developing new mathematical tools and the corresponding language in order to reformulate the {\bf cobordism hypothesis} within this framework and to outline its proof.  

The main theorem for the {\bf oriented two-dimensional cobordism hypothesis} was given by Ch. Schommer-Pries~\cite{SP2014}, which we recall in a combined form with the succinct formulation~[\cite{HSchV2016} Theorem 2.10].

\begin{thm}[\cite{SP2014} Theorem 3.52 ]
Let $k$ be an algebraically closed field with $\operatorname{char}(k)=0$.
 There exists a weak $2$-functor:
\begin{eqnarray*}
\operatorname{Fun}_{\otimes}(\operatorname{Cob}^{\operatorname{or}}_{2,1,0},\mathbf{Alg}^2_k)&\rightarrow&\mathbf{sFrob}_k\\
Z&\mapsto&Z(\bullet_+),
\end{eqnarray*}
i.e. there exists an equivalence of the bicategory of two-dimensional oriented fully extended TFTs with values in $\mathbf{Alg}_k^2$ and the bicategory of semi-simple Frobenius algebras $\mathbf{sFrob}_k$.
\end{thm}

In String Theory, independently M. Fukuno, S. Hosono and H. Kawai~\cite{FHK1994} and C. Bachas and M. Petropolous~\cite{BP1993} investigated {\bf lattice TFTs} (LTFT) and constructed {\bf state sums} based on triangulations of ordinary two-dimensional cobordisms. Their results can be rephrased as follows. 
\begin{thm}[\cite{BP1993,FHK1994}]
Let $k$ be an algebraically closed field with $\operatorname{char}(k)=0$.
The class of LTFTs is equivalent to $\mathbf{fdAlg}_k^2$, i.e. the fully dualisable objects in $\mathbf{Alg}_k^2$. For  $A\in\mathbf{fdAlg}_k^2$, the centre $\mathfrak{z}(A)$, corresponds to the closed string states and $\dim_k(\mathfrak{z}(A))$ corresponds to the number of independent physical operators. 
\end{thm}
Before we proceed further, let us summarise the chain of equivalences, which results from \cite{HSchV2016,Lu2009,SP2014}: 
\begin{equation*}
\mathbf{fdAlg}_k^2\leftrightarrow\mathbf{sFrob}_k\leftrightarrow\text{Calabi-Yau category}.
\end{equation*}

The following relation holds between rational languages and closed string theory. 
\begin{thm}
The category of oriented $2$-dimensional TQFTs is equivalent to the $s$-variety of exchangeable rational power series.
\end{thm}
\begin{proof}
The classic equivalence between $2D$ TQFTs and $\mathbf{cFrob}_k$, cf. e.g.~\cite{K2003}, combined with Corollary~\ref{ex_cFrob}, yields the statement.  
\end{proof}

The second relation between rational series and open-closed string theory is given next.
\begin{thm}
Let $X$ be a finite alphabet and $C$ a regular biprefix code. Then the following correspondences hold:
\[
\begin{xy}
  \xymatrix{
      \{\text{$C$ code: regular, biprefix}\}\ar[r]\ar[d]_{\mathfrak{z}} & \mathbf{fdAlg}_k^2 \ar[r] \ar[d]^{\mathfrak{z}} &\text{LTFT}\ar[d]^{\mathfrak{z}}\ar[l] \\
                            \{\text{$s$-variety: rational exchangeable series}\}\ar[r]   &\mathbf{cFrob}_k \ar[r]\ar[l]&\text{$2D$ TQFT}\ar[l]
               }
\end{xy}
\]
where $\mathfrak{z}$ is the {\bf centraliser}, and $\mathbf{cFrob}$ is given by the centres $\mathfrak{z}(A)$ of the finite-dimensional semi-simple algebras $A$.
\end{thm}
\subsection{Group algebras}
Group algebras of finite groups have been of particular interest in string theory, cf. e.g.~\cite{BP1993,FHK1994, Kim2017,LPf2006,M2007}.

The content of [\cite{R1982} Remarks:~1. p.455] can be stated as follows.

For $G\in\mathbf{fGr}$, a finite group, let $X:=|G|$ be the underlying set (alphabet) with elements (letters) $|g|$. Let $j_G:k\langle |G|\rangle\rightarrow G$ be the canonical monoid epimorphism given by $|g|\mapsto g$, and define 
\begin{equation}
\label{group_code}
C^*_G:=j_G^{-1}(e),
\end{equation}
where $e\in G$ is the neutral element.
\begin{prop}[\cite{R1982}]
Let $G\in\mathbf{fGrp}$ and $C^*_G$ as in~(\ref{group_code}). Then $C^*_G$ is generated by a rational biprefix code $C_G$ with the syntactic algebra $\mathfrak{A}_{C^*_G}$ being isomorphic to $k[G]$, where $k[G]$ is the group algebra of $G$, i.e. we have the commutative diagram
\[
\begin{xy}
  \xymatrix{
      G\ar[r]\ar[d] & \text{$C_G$ code: rational, biprefix} \ar[d] \\
                            k[G]\ar[r]  & \mathfrak{A}_{C^*_G}:\text{semisimple, finite-dimensional}
               }
\end{xy}
\]
\end{prop}
We have the following statement.
\begin{prop}
Let $k$ be an algebraically closed field of characteristic zero.
The $k$-functor $F:\mathbf{fGrp}\rightarrow\operatorname{LTFT}_k$ factorises through the category of rational biprefix codes, i.e. the diagram
\[
\begin{xy}
  \xymatrix{
      \mathbf{fGrp}\ar[r]\ar[d] &\operatorname{LTFT}_k \\
                            \{\text{code: rational, biprefix}\}\ar[ur]   &                }
\end{xy}
\]
commutes.
\end{prop}
\begin{rem}
Let $A\in\mathbf{Frob}$ such that $\lambda(1_A)=1_k$. Then $(A,\lambda)$ defines a non-commutative probability space. In particular, $(k[G],\phi)$ is a non-commutative probability space with a faithful trace $\phi$. Further, if we restrict to permutation groups then we obtain a natural relation with free probability theory, cf.~\cite{F2018}. 
\end{rem}

\section{Two-cobordism and second order monadic logic}

Here we show that the rational series describing the algebras which are equivalent to the topological field theories are describable by a restricted version of monadic second order logic (MSO). This is possible by an extension of the {\bf Büchi-Elgot Theorem} given by M. Droste and P. Gastin~\cite{DG2009}. For the necessary facts cf.~\cite{BZ2012,DG2009,Lib2004} and in particular for the definition of 
{\bf restricted weighted second order monadic logic} (rwMSO)~\cite{DG2009} or the lecture notes~\cite{BZ2012}.
\begin{thm}[\cite{DG2009}]
Let $X$ be a finite alphabet and $\operatorname{char}(k)=0$. Then for $S\in k\langle\langle X\rangle\rangle$ the following equality holds:
$S$ is $\operatorname{rwMSO}(k,X)$-definable iff $S$ is rational (recognisable).
\end{thm}
\begin{cor}
Let $X$ be a finite alphabet and $S\in k\langle X\rangle^*$ a formal power series. Then the following equivalences hold:
\begin{equation*}
\dim_k(\mathfrak{A}_S)<\infty\Leftrightarrow \text{$S$ is rational}\Leftrightarrow\text{$S$ is $\operatorname{rwMSO}(k,X)$ definable}.
\end{equation*}
\end{cor}
The relevance with respect to two-dimensional lattice topological field theories is given by:
\begin{prop}
Every LTFT is $\operatorname{rwMSO}(k)$-definable, i.e. let $A\in\mathbf{fdAlg}_k^2$ be the finite-dimensional semi-simple algebra corresponding to an LTFT and $(X,S)$ a formal power series with syntactic algebra $\mathfrak{A}_S=A$. Then $S$ is $\operatorname{rwMSO}(k)$-definable.
\end{prop}
Let us conclude with following observations and remarks. 

The relations we have established in a first step between weighted finite-state automata, second order monadic logic and fully extended two-dimensional topological quantum field theories are at an algebraic level. However, the automata theoretic but also model theoretic part can be described in more intrinsic, i.e. (higher) categorical, terms which is necessary in order to extend and generalise the present results. 

Further, it appears that the relation between logic and $2D$ TFTs should generalise to higher dimensions. Namely, the order of the logic / type theory should parallel the dimension of the cobordisms involved, i.e. we have:
$$
\text{$n$-cobordism $\leftrightarrow n$-order logic/type theory}.
$$

\section*{Acknowledgments}
I would like to thank Owen Gwilliam, Claudia Scheimbauer, Chris Schommer-Pries and Urs Schreiber for fruitful conversations which helped me to acquire the necessary background and for giving me a generous perspective on various aspects of the matters involved and their own related work. I would also like to thank John McKay for his interest and encouragement and Roland Speicher for his support, the conversations during the time as it evolved and his interest. Finally, I would like to thank the MPI in Bonn for its hospitality. This work was supported, at its beginning, by the MPI and presently by the ERC advanced grant ``Noncommutative distributions in free probability".


\begin{thebibliography}{00}
\bibitem{BP1993} C. Bachas, P.M.S. Petropoulos, {\it Topological Models on the Lattice
and a Remark on String Theory Cloning}, Commun. Math. Phys. {\bf 152},191-202 (1993)
\bibitem{BP2012} V. Barbu , T. Precupanu, {\it Convexity and Optimization in Banach Spaces}, 4th edition, Springer Monographs in Mathematics, Springer 2012
\bibitem{BPR2009} J. Berstel, D. Perrin, Ch. Reutenauer, {\it Codes and Automata}, Encyclopedia of Mathematics and its Applications, {\bf 129}, Cambridge University Press, (2009)
\bibitem{BR} J. Berstel, Ch. Reutenauer, {\it  Rational Series and
Their Languages}, EATCS, Monographs in Theoretical Computer Science. Springer 1988
\bibitem{BZ2012} B. Bollig, M. Zeitoun, {\it Weighted Automata}, Lecture Notes, Version of May 15, (2012)
\bibitem{DG2009} M. Droste, P. Gastin, {\it Weighted automata and weighted logics}. In: Handbook of Weighted Automata. Ed. by W. Kuich, H. Vogler, and M. Droste. EATCS Monographs in Theoretical Computer Science. Springer, (2009)
\bibitem{Fl1974} M. Fliess, {\it Matrices de Hankel}, J. Math. pures et appl., {\bf 53},  p. 197-224, 1974
\bibitem{FK2011} R. Friedrich and T. Kato, {\it The Mesoscopic category, Automata and Tropical Geometry}, arXiv:1111.2832v1 (2011)
\bibitem{F2018} R. Friedrich,  preprint (2018)
\bibitem{FHK1994} M. Fukuma, S. Hosono, H. Kawai, {\it Lattice topological field theory in two dimensions}, Comm. Math. Phys. {\bf 161}, 157-175, (1994)
\bibitem{HSchV2016} J. Hesse, C. Schweigert, A. Valentino {\it An equivalence between semisimple symmetric Frobenius algebras and Calabi-Yau categories}, Theory Appl. Categ. {\bf 32}, No. 18, (2017)
\bibitem{Kim2017} Y. Kimura, {\it Noncommutative Frobenius algebras and open-closed duality}, arXiv:hep-th/1701.8382, (2017)
\bibitem{K2003} J. Kock, {\it Frobenius algebras and 2D topological quantum field theories}, No. {\bf 59} of LMSST, Cambridge University Press, (2003). 
\bibitem{L1999} T.-Y. Lam, {\it Lectures on Modules and Rings}, Graduate Texts in Mathematics, Springer 1999
\bibitem{LPf2006}A. Lauda, H. Pfeiffer, {\it State sum construction of two-dimensional open-closed Topological Quantum Field Theories}, arXiv:math/0602047v2 (2006)
\bibitem{Lib2004} L. Libkin, {\it Elements of Finite Model Theory},  Texts in Theoretical Computer Science, An EATCS Series, Springer 2004
\bibitem{Lu2009} J. Lurie, {\it On the classification of topological field theories}, Current developments in mathematics, Int. Press, Somerville, MA, 2009, pp. 129-280. arXiv:0905.0465 (2009)
\bibitem{M2007} J. Morton, {\it Extended TQFT's and Quantum Gravity}
\bibitem{R1980} Ch. Reutenauer, {\it Séries formelles et algèbres syntactiques},  Journal of Algebra, {\bf 66}, 448-483, 1980
\bibitem{R1982} Ch. Reutenauer, {\it Biprefix codes and semisimple algebras}, In: Nielsen M., Schmidt E.M. (eds) Automata, Languages and Programming. ICALP 1982. Lecture Notes in Computer Science, vol 140. Springer, Berlin, Heidelberg 1982
\bibitem{SP2014} C. Schommer-Pries, {\it The classification of two-dimensional extended topological field theories}, arXiv:1112.1000 (2014)
\end{thebibliography}
\end{document}